\documentclass[a4paper]{amsart}
\usepackage{amssymb,amsmath}
\usepackage{mathrsfs} 
\usepackage{enumitem}
\usepackage[capitalize]{cleveref}
\newcounter{alph}

\newtheorem{theo}[alph]{Theorem}

\numberwithin{equation}{section}
\newtheorem{cor}[equation]{Corollary}
\newtheorem{lem}[equation]{Lemma}
\newtheorem{prop}[equation]{Proposition}
\newtheorem{thm}[equation]{Theorem}
\theoremstyle{definition}
\newtheorem{exa}[equation]{Example}
\newtheorem{rem}[equation]{Remark}

\def\N{\mathbb N}
\def\R{\mathbb R}
\def\Z{\mathbb Z}
\def\H{\mathcal H}
\def\P{\mathcal P}
\def\ve{\varepsilon}
\def\vf{\varphi}
\def\la{\langle}
\def\ra{\rangle}
\newcommand{\diam}{\operatorname{diam}}
\newcommand{\Hom}{\operatorname{Hom}}
\newcommand{\grad}{\operatorname{grad}}

\begin{document}


\title[Diffusions, discretizations, harmonic functions]
{Equivariant discretizations of diffusions and harmonic functions of bounded growth}
\author{Werner Ballmann}
\address
{WB: Max Planck Institute for Mathematics, Vivatsgasse 7, 53111 Bonn.}
\email{hwbllmnn@mpim-bonn.mpg.de}
\author{Panagiotis Polymerakis}
\address
{PP: Max Planck Institute for Mathematics, Vivatsgasse 7, 53111 Bonn.}
\email{polymerp@mpim-bonn.mpg.de}

\thanks{We would like to thank
J\"urgen Jost for pointing out the topic of harmonic functions of bounded growth
and Fran\c{c}ois Ledrappier for helpful comments.
We would also like to thank the Max Planck Institute for Mathematics and the Hausdorff Center for Mathematics in Bonn for their support and hospitality.}

\date{\today}


\begin{abstract}
For covering spaces and properly discontinuous actions with compatible diffusion operators,
we discuss Lyons-Sullivan discretizations of the associated diffusions and harmonic functions of bounded growth.
\end{abstract}

\maketitle

\section{Introduction}
\label{intro}

We are interested in spaces of harmonic functions of bounded growth.
This topic got started with the work of Yau on harmonic functions on Riemannian manifolds
and his conjecture, solved by Colding and Minicozzi,
that the spaces $\H^d(M)$ of harmonic functions
of polynomial growth of degree at most $d\ge0$ on a complete Riemannian manifold $M$
with non-negative Ricci curvature are of finite dimension \cite{Y1,Y2,CM}.

We consider a  non-compact and connected manifold $M$
together with an (elliptic) diffusion operator $L$ on $M$ that is symmetric on $C^\infty_c(M)$
with respect to a smooth volume element on $M$ (see \cref{suselap}).
The reader not familiar with diffusion operators should think of the Laplacian on Riemannian manifolds.
We are interested in two related scenarios.
In the first, we are given a cocompact covering $p\colon M\to M_0$
and assume that $L$ and the volume element on $M$
are the pull-backs of a diffusion operator $L_0$ and a smooth volume element on $M_0$.
In the second,
we are given a properly discontinuous and cocompact action on $M$ by a group $\Gamma$
and assume that $L$ and the volume element on $M$ are $\Gamma$-invariant.
To avoid case distinctions, we consider an orbifold covering $p\colon M\to M_0$,
where the manifold $M$ is considered with the trivial orbifold struture,
$M_0$ is a closed orbifold, and $L$ and the volume element on $M$
are the pull-backs of a diffusion operator $L_0$ and a smooth volume element on $M_0$.
This setup contains the above two scenarios,
where $M_0$ is the orbit space $\Gamma\backslash M$ in the second scenario.
The Riemannian metric on $M$ associated to $L$ is the pull-back of the Riemannian metric
on $M_0$ associated to $L_0$ and is therefore complete.

Our main results establish a one-to-one correspondence between
$L$-harmonic functions of bounded growth on $M$
and $\mu$-harmonic functions of bounded growth on a given fiber $X\subseteq M$ of $p$,
where $\mu$ belongs to a certain class of families $\mu=(\mu_y)_{y\in M}$ of probability measures on $X$
and where $\mu$-harmonic functions on $X$ are the solutions of the operator $\Delta_\mu$ defined by
\begin{align}
	(\Delta_\mu f)(y) = \sum_{x\in X} \mu_y(x)(f(x)-f(y)).
\end{align}
The classes of $\mu$ used here have their origin in work of Furstenberg \cite[Section 5]{Fu},
were introduced and studied by Lyons and Sullivan \cite[Sections 7 and 8]{LS},
and later refined in \cite[Sections 1 and 2]{BL2}.
For the case of diffusion operators as considered here, they are discussed in \cite[Section 3]{BP}.
We refer to them as LS-measures.
They depend on the choice of data, refered to as LS-data (see \cref{seclsd}).

We say that a function $a\colon[0,\infty)\to\R$ is a \emph{growth function}
if it is monotonically increasing, if $a(0)\ge1$,
and if $a$ is \emph{submultiplicative} in the sense that, for all $r,s\ge0$,
\begin{align}\label{submul}
	a(r+s)\le C_a a(r)a(s).
\end{align}
Besides the constant function $1$,
the functions $(r+1)^\alpha$ and $e^{\alpha r}$ with $\alpha>0$
are the most important growth functions
and give rise to the concepts of polynomial and exponential growth.
Another interesting class are the functions $e^{cr^\alpha}$ with $c>0$ and $0<\alpha<1$,
which are between polynomial and exponential growth.

\begin{exa}\label{fgg}
Let $S$ be a finite and symmetric generating set of a group $\Gamma$
and $N_S(m)$ be the number of elements of $\Gamma$
which can be expressed as a word in $S$ of length at most $m\in\N_0$.
Then $N_S$ is monotonically increasing with $N_S(0)=1$ and $N_S(m+n)\le N_S(m)N_S(n)$.
Since $\lfloor r+s\rfloor\le\lfloor r\rfloor + \lfloor s\rfloor+1$,
$a=a(r)=N_S(\lfloor r\rfloor)$ is a growth function with $C_a=N_S(1)$.
\end{exa}

Replacing $a$ by the function $C_aa$,
the constant $C_a$ in \eqref{submul} disappears.
We say that a growth function $a$ is \emph{subexponential} if
\begin{align*}
	\lim_{r\to\infty}\frac1r\ln a(r) = 0.
\end{align*}
The above functions $(r+1)^\alpha$ with $\alpha>0$ and $e^{cr^\alpha}$ with $c>0$ and $0<\alpha<1$
are examples of subexponential growth functions.

\subsection{Main results}
\label{submai}
We let $X$ be a fiber of $p$, fix an origin $x_0\in X\subseteq M$, and set $|x|=d(x,x_0)$.
For a growth function $a$,
we say that a function $f$ on $M$ or $X$ is \emph{$a$-bounded}
if there is a constant $C_f\ge1$ such that
\begin{align}\label{aboun}
	|f(x)| \le C_f a(|x|)
\end{align}
for all $x\in M$ or $x\in X$, respectively.
By \eqref{submul} and the triangle inequality,
whether or not a function on $M$ or $X$ is $a$-bounded
does not depend on the choice of $x_0$.

We denote by $\H_a(M,L)$ and $\H_a(X,\mu)$ the spaces of $a$-bounded $L$-harmonic functions on $M$ and $a$-bounded $\mu$-harmonic functions on $X$, respectively.
Clearly
\begin{align*}
	\H_a(M,L) \subseteq \H_b(M,L)
	\quad\text{and}\quad
	\H_a(X,\mu) \subseteq \H_b(X,\mu)
\end{align*}
for any two growth functions $a$ and $b$ such that $a\le cb$ for some constant $c>0$.

Our first main result is known in the case of bounded harmonic functions, that is,
for the function $a=1$; see \cite[Theorem 1.11]{BL2} or the earlier \cite[Theorem 1]{Ka2}.

\begin{theo}\label{thmbg}
Suppose that $a$ is a subexponential growth function
and that the LS-data for the LS-measures are appropriately chosen.
Then the restriction of an $a$-bounded $L$-harmonic function on $M$ to $X$
is $a$-bounded and $\mu$-harmonic,
and the restriction map $\H_a(M,L)\to\H_a(X,\mu)$ is an isomorphism.
\end{theo}

The precise meaning of the term \lq appropriate\rq\ will be made clear in the text.
In the two setups we consider, appropriate choices of LS-data are always possible,
but are far from being unique.

In the discussion of asymptotic properties of geometric objects,
quasi-isometries play a central role.
Now with respect to a quasi-isometry, an $a$-bounded function is $b$-bounded,
where $b(r)=a(cr)$ for some suitable constant $c\ge1$.
According to this,
we say that two growth functions $a$ and $b$ belong to the same \emph{growth type}
if there is a constant $c\ge1$ such that
\begin{align*}
    a(r/c)/c \le b(r) \le ca(cr)
\end{align*}
for all $r\ge0$.
Clearly, growth types partition the space of growth functions.
Moreover, the property of being subexponential depends only on the type.

Given a growth type $A$, we say that a function $f$ on $M$ or $X$ is $A$-bounded if,
for one or, equivalently, for any $a\in A$, there is a  constant $C_f\ge1$ such that
\begin{align}\label{Aboun}
	|f(x)| \le C_f a(C_f|x|)
\end{align}
for all $x\in M$ or $x\in X$, respectively.

We denote by $\H_A(M,L)$ and $\H_A(X,\mu)$ the spaces of $A$-bounded $L$-harmonic functions on $M$
and $A$-bounded $\mu$-harmonic functions on $X$, respectively.
Our second main result is an immediate consequence of \cref{thmbg}.

\begin{theo}\label{corbg}
Suppose that $A$ is a subexponential growth type
and that the LS-data for the LS-measures are appropriately chosen.
Then the restriction of an $A$-bounded $L$-harmonic function on $M$ to $X$
is $A$-bounded and $\mu$-harmonic,
and the restriction map $\H_A(M,L)\to\H_A(X,\mu)$ is an isomorphism.
\end{theo}

\subsection{Applications}
\label{subap}
We discuss three applications of our results to the case of $L$-harmonic functions of polynomial growth,
that is, the growth types determined by the growth functions $(r+1)^d$, $d\ge1$.
The solution of Yau's conjecture by Colding-Minicozzi \cite{CM},
Gromov's theorem on groups of polynomial growth \cite{Gr},
and the work of Kleiner \cite{Kl} on Gromov's theorem and on harmonic functions of polynomial growth
belong to the background of our discussion.

We assume throughout that $M$ is non-compact and connected
and that the diffusion operator $L$ on $M$ and the volume element
are invariant under a group $\Gamma$,
which acts properly discontinuously and cocompactly on $M$.
Recall that $\Gamma$ is then finitely generated.

We are interested in the spaces $\H^d(M,L)$ of $L$-harmonic functions
of polynomial growth of degree at most $d$, that is, $L$-harmonic functions $h$ on $M$ such that
\begin{align}\label{senod}
	\|h\|_d = \limsup_{|x|\to\infty}\frac{|h(x)|}{|x|^d} < \infty.
\end{align}
The space of bounded $L$-harmonic functions is then written as $\H^0(M,L)$, and we have
\begin{align*}
	\H^0(M,L) \subseteq \H^1(M,L) \subseteq \H^2(M,L) \subseteq \dots
\end{align*}
It is well known and easy to see that $H^0(M,L)$ consists either of constant functions only,
and then $\dim\H^0(M,L)=1$, or that $\dim\H^0(M,L)=\infty$.
By \cite[Theorem 3]{LS}, the latter holds if $\Gamma$ is not amenable.
We discuss $\H^d(M,L)$ for $d\ge1$.
Our strategy consists of combining results of Meyerovitch, Perl, Tointon, and Yadin \cite{MPTY,MY,Pe}
about $\mu$-harmonic functions on groups and translating them using \cref{corbg}.
More detailed references will be given in the text.

In the proofs of the first two of our applications, Theorems \ref{cornil} and \ref{thmnil},
we also use work of Kuchment and Pinchover \cite{KP} on harmonic functions of Schr\"odinger operators
in the case where $\Gamma$ contains $\Z$ or $\Z^2$ as a subgroup of finite index,
due to a symmetry question concerning LS-measures.
Via renormalization as discussed in \cref{suselap},
their \cite[Theorem 5.3]{KP} on Schr\"odinger operators actually implies Theorems \ref{cornil} and \ref{thmnil}
in the case where $\Gamma$ is almost Abelian. 

Our first application is related to a special case of a Liouville theorem of Cheng,
namely that a harmonic function on a complete Riemannian manifold of non-negative Ricci curvature
is bounded if it is of sublinear growth \cite[p.\,151]{Ch}.

\begin{theo}\label{cornil}
If $\Gamma$ is virtually nilpotent and $h$ is a harmonic function on $M$ of polynomial growth,
then the growth of $h$ is integral.
More precisely, if $h\in \H^d(M,L)$ for some integer $d\ge1$,
then $\|h\|_d$ is either positive or else $h\in\H^{d-1}(M,L)$.
\end{theo}

For $g\in\Gamma$ and a function $f$ on $\Gamma$,
we define the \emph{partial derivative $\partial_gf$} by
\begin{align*}
	\partial_gf(h) = f(gh) - f(h). 
\end{align*}
We say that $f$ is a \emph{polynomial of degree at most $d$}
if all iterated partial derivatives
\begin{align*}
	\partial_{g_0}\cdots\partial_{g_d}f
\end{align*}
of $f$ vanish for all $d+1$ elements $g_0,\dots,g_d\in\Gamma$
and denote by $\P^d(\Gamma)$ the space of all such polynomials
(with the convention $\P^d(\Gamma)=\{0\}$ for $d<0$).

\begin{exa}\label{exapol}
Consider the free Abelian group $\Gamma=\Z^k$.
Clearly, with respect to the usual inclusion $\Z^k\subseteq\R^k$,
any polynomial on $\Z^k$ of degree at most $d$
is the restriction of a polynomial of degree at most $d$ on $\R^k$.
Thus restriction defines an isomorphism $\P^d(\R^k)\to\P^d(\Z^k)$.
\end{exa}

Since $\Gamma$ is finitely generated,
$\P^d(\Gamma)$ is of finite dimension for any $d\ge0$ \cite[Proposition 1.15]{Le}.
In fact, there is a recursive schema for its dimension in terms of the lower central series of $\Gamma$ \cite[Proposition 1.10]{MPTY}.

\begin{exa}\label{exacoh}
By definition, $\P^0(\Gamma)$ is equal to the space of constant real valued functions on $\Gamma$
so that $\P^0(\Gamma)\cong\R$.
Furthermore, $\P^1(\Gamma)$ consists of affine real valued functions on $\Gamma$,
so that $\P^1(\Gamma)\cong\Hom(\Gamma,\R)\oplus\P^0(\Gamma)$.
In particular, $\dim\P^1(\Gamma)-1=b_1(\Gamma,\R)$,
the first Betti number of $\Gamma$ with respect to real coefficients.
\end{exa}

Our second application is the following version of \cite[Corollary 0.10]{CM} of Colding-Minicozzi
and \cite[Theorem 1.3]{Kl} of Kleiner.

\begin{theo}\label{thmnil}
If $\Gamma$ is virtually nilpotent,
then $\H^d(M,L)$ is finite-dimensional for all $d\ge0$.
More precisely, if $N\subseteq\Gamma$ is a nilpotent subgroup of finite index,
then
\begin{align*}
	\dim\H^d(M,L) = \dim\P^d(N) - \dim \P^{d-2}(N)
\end{align*}
for all $d\ge0$.
In particular, $\dim\H^d(M,L)$ does not depend on $L$.
\end{theo}

\begin{exa}\label{exapol2}
In the situation of \cref{thmnil}, consider the case where $N\cong\Z^k$.
From \cref{exapol}, we get that
\begin{align*}
	\dim\H^d(M,L)
	&= \dim\P^d(\Z^k) - \dim \P^{d-2}(\Z^k) \\
	&= \textstyle{\binom{k+d}{k} - \binom{k+d-2}{k}}
	= \textstyle{\frac{k+2d-1}{k+d-1}\binom{k+d-1}{k-1}}.
\end{align*}
\end{exa}

\begin{exa}\label{exanopo}
If $M$ is simply connected and the sectional curvature of $M$ is non-positive,
then either $\Gamma$ contains a subgroup isomorphic to the free group $F_2$,
or else $M$ is isometric to Euclidean space $\R^m$,
where $m=\dim M$ \cite[Theorem A]{BE}.
In the first case, $\Gamma$ is non-amenable
and then $\H^0(M,L)$ is infinite dimensional,
therefore also all $\H^d(M,L)$ with $d\ge1$.
In the second case, $\Gamma$ contains $\Z^m$ as a subgroup of finite index,
and we are in the context of \cref{exapol2}.
\end{exa}

\begin{rem}
Extending and refining an earlier estimate of Hua and Jost \cite[Theorem 1.1]{HJ},
Meyerovich et al.\ \cite[Corollary 1.12]{MPTY} obtain that
\begin{align*}
	c_1d^r \le \dim\H^d(\Gamma,\mu) \le c_2d^r
\end{align*}
for all $d\ge1$, where $\mu$ is a \emph{courteous probability measure} on $\Gamma$ in the sense of \cite{MY},
$c_1<c_2$ are positive constants, and $r$ is the rank of the nilpotent subgroup $N\subseteq\Gamma$
of finite index.
Here we use \cite[Corollary 1.9]{MPTY} and \cite[Theorem 1.5]{Pe} to pass from finitely supported,
symmetric probability measures $\mu$ on $\Gamma$, whose support generates $\Gamma$,
as assumed in \cite[Corollary 1.12]{MPTY}, to the more general class of courteous probability measures.
This class includes the probability measures on $\Gamma$ induced from LS-measures as used here,
at least in the case where the LS-data are appropriately chosen
and $\Gamma$ does not contain $\Z$ or $\Z^2$ as a subgroup of finite index. 
\end{rem}

Finally, we have the following version of a result of Meyerovitch and Yadin \cite[Theorem 1.4]{MY}.

\begin{theo}\label{thmsol}
If $\Gamma$ is virtually solvable, then the following are equivalent:
\begin{enumerate}
\item
$\Gamma$ is virtually nilpotent;
\item
$\dim\H^d(M,L) < \infty$ for some $d\ge1$;
\item
$\dim\H^1(M,L) < \infty$.
\end{enumerate}
\end{theo}

\begin{exa}\label{exasol}
If $\Gamma$ is linear,
then either $\Gamma$ contains a subgroup isomorphic to the free group $F_2$,
or else $\Gamma$ is virtually solvable, by the Tits alternative.
In the first case, $\H^0(M,L)$ is infinite dimensional, hence also $\H^1(M,L)$,
in the second, \cref{thmsol} applies.
Hence, if $\Gamma$ is linear,
the assertions of \cref{thmsol} hold without assuming that $\Gamma$ is virtually solvable.
\end{exa}

Since $\Gamma$ is finitely generated and the probability measure $\mu$ on $\Gamma$
induced from LS-measures as used here satisfy the properties required in \cite{MPTY,MY},
at least if the LS-data are chosen appropriately
and $\Gamma$ does not contain $\Z$ or $\Z^2$ as a subgroup of finite index (see \cref{court}),
$\H^d(M,L)$ is conjecturally finite dimensional for some (or any) $d\ge1$
if and only if $\Gamma$ is virtually nilpotent; compare with the introductions to \cite{MPTY,MY}.

\subsection{Laplace-type operators and renormalization}
\label{suselap}
With respect to the Riemannian metric associated to a diffusion operator $L$ on a manifold $M$,
we have $L=\Delta+Y$, where $Y$ is a smooth vector field on $M$,
and, conversely, any operator of that form is a diffusion operator.
More generally, if $M$ is Riemannian,
a differential operator $L$ on $M$ is said to be of \emph{Laplace-type} if it is of the form
\begin{align}
  L = \Delta + Y + V,
\end{align}
where $Y$ is a smooth vector field and $V$ a smooth function on $M$,
the \emph{drift vector field} and \emph{potential} of $L$.
In this notation,
$L$ is symmetric on $C^\infty_c(M)$ with respect to a smooth volume element $\vf^2{\rm dv}$,
where $\vf>0$, if and only if $Y=-2\grad\ln\vf$.
The orthogonal isomorphism
\begin{align*}
  m_{\vf} \colon L^2(M,\vf^2{\rm dv}) \rightarrow L^2(M,{\rm dv}), \quad m_{\vf} f = \vf f
\end{align*}
transforms $L$ then into the Schr\"odinger operator
\begin{align*}
	S = m_\vf \circ L \circ m_{\vf}^{-1} = \Delta+(V-\Delta\vf/\vf) f,
\end{align*}
which is symmetric on $C^\infty_c(M)$ with respect to ${\rm dv}$.
We refer to this transformation as \emph{renormalization (with $1/\vf$)}.

Using renormalization, the above results, properly formulated,
also hold for Laplace-type operators.
More precisely, 
for the Schr\"odinger operator $S$ as above,
we let $\psi$ be the lift of a positive eigenfunction $\psi_0$
of the corresponding Schr\"odinger operator $S_0$ on $M_0$ with respect
to the bottom $\lambda_0=\lambda_0(M_0,S_0)$ of the spectrum of $S_0$ on $M_0$.
(For analysis on orbifolds, see e.g.\,\cite{Fa}.)
Renormalizing a second time, now $S-\lambda_0$ with $\psi$,
yields the diffusion operator $L'=\Delta-2\grad\ln\psi$ on $M$,
which is symmetric with respect to the smooth volume element $\psi^2{\rm dv}$.
Thus multiplication with $\vf/\psi$ induces a bijection between the spaces of
$(L-\lambda_0)$-harmonic functions and $L'$-harmonic functions.
Since $\vf/\psi$ is bounded between two positive constants,
growth properties of functions are stable under multiplication with $\vf/\psi$.
In conclusion, if $a$ is of subexponential growth,
then multiplication with $\vf/\psi$ followed by restriction to $X$ yields isomorphisms
\begin{align*}
  \H_a(M,L-\lambda_0) \to  H_a(X,\mu)
  \quad\text{and}\quad
  \H_A(M,L-\lambda_0) \to  H_A(X,\mu),
\end{align*}
by what we just said and Theorems \ref{thmbg} and \ref{corbg}.
The results corresponding to the ones in  \cref{subap} are immediate consequences.
We note that here,
by the amenability of the group $\Gamma$ in Theorems \ref{cornil}, \ref{thmnil}, and \ref{thmsol},
$\lambda_0(M_0,S_0)=\lambda_0(M,S)$, the bottom of the spectrum of $S$ on $M$,
at least if the action of $\Gamma$ on $M$ is also free \cite{BMP,BC}.

\subsection{Structure of the article}
\label{susestruc}
In \cref{seclsd},
we present the Lyons-Sullivam discretization of the $L$-diffusion in the way we need it,
recall several results about it from the literature,
and prove that the LS-measures have finite exponential moments.
The third section constitutes the heart of the paper.
We show an extended version of \cref{thmbg}
in the case where $L$ is invariant under a group $\Gamma$
which acts properly discontinuously and cocompactly on $M$.
The extension to the orbifold case is contained in the fourth section.
In the short final section, we prove Theorems \ref{cornil}--\ref{thmsol}.

\tableofcontents

\section{Lyons-Sullivan discretization of diffusions}
\label{seclsd}
Following earlier work of Furstenberg,
Lyons and Sullivan (LS) constructed a discretization of Brownian motion on Riemannian manifolds \cite{LS}.
The LS-construction was taken up and refined in \cite{BL2}.
It actually applies also to diffusions associated to (elliptic) diffusion operators,
and that extension was described in \cite{BP}.
We start this section with an outline of the LS-construction for such diffusions.
The main new results are in \cref{susemom}.

Let $L$ be a diffusion operator on a connected manifold,
and assume that the $L$-diffusion on $M$,
that is, the diffusion with generator $L$, is complete.
Let $\Omega$ be the space of paths $\omega\colon[0,\infty)\to M$,
endowed with the compact-open topology.
For $x\in M$, denote by $P_x$ the probability measure on $\Omega$
corresponding to starting the $L$-diffusion at $x$.
For a measure $\mu$ on $M$, set $P_\mu=\int_M\mu(dx)P_x$.

\subsection{Balayage and $L$-harmonic functions}
\label{subset}
Let $F\subseteq M$ be closed and $V\subseteq M$ be open.
For $\omega\in\Omega$, the respective \emph{hitting} and \emph{exit time},
\begin{equation}\label{rfsv}
\begin{split}
  R^F(\omega) &= \inf \{ t\ge0 \mid \omega(t)\in F \}, \\
  S^V(\omega) &= \inf \{ t\ge0 \mid \omega(t) \in M\setminus V\},
\end{split}
\end{equation}
are stopping times.
For a measure $\mu$ on $M$ and a Borel subset $A\subseteq M$, let
\begin{equation}
\begin{split}
  \beta(\mu,F)(A) &= \beta_\mu^F(A) = P_\mu(\omega(R^F(\omega))\in A), \\
  \ve(\mu,V)(A) &= \ve_\mu^V(A) = P_\mu(\omega(S^V(\omega))\in A),
\end{split}
\end{equation}
where $\beta$ stands for \emph{balayage} and $\ve$ for \emph{exit}.
In the case of Dirac measures, $\mu=\delta_x$, we use the shorthand $x$ for $\delta_x$.
If $\beta(x,F)(F)=1$ for all $x\in M$, then $F$ is said to be \emph{recurrent}.
This is equivalent to $R^F<\infty$ almost surely with respect to each $P_x$.

\begin{prop}\label{swel}
Let $F$ be a recurrent closed subset of $M$, $\mu$ a finite measure on $M$, and $h\colon M\to\R$ an $L$-harmonic function.
Then we have:
\begin{enumerate}
\item\label{swelb}
If $h$ is bounded, then $\mu(h)=\beta(\mu,F)(h)$.
\item\label{swelp}
If $h$ is positive, then $\beta(\mu,F)(h)\le\mu(h)$.
\end{enumerate}
\end{prop}

An $L$-harmonic function $h$ on $M$ is said to be \emph{swept by $F$}
if $\beta(x,F)(h)=h(x)$ for all $x\in M$.
Then
\begin{align}\label{swee}
	\mu(h)=\beta(\mu,F)(h)
\end{align}
for all finite measures $\mu$ on $M$.
By \cref{swel}.\ref{swelb}, any bounded $L$-harmonic function is swept by any recurrent closed subset of $M$.

\subsection{LS-discretization and $L$-harmonic functions}
\label{sublsdh}
Let $X$ be a discrete subset of $M$.
Families $(F_x)_{x\in X}$ of compact subsets and $(V_x)_{x\in X}$ of relatively compact open subsets of $M$
together with a constant $C>1$  will be called \emph{regular Lyons-Sullivan data for $X$}
or, for short,  \emph{regular LS-data for $X$} if
\begin{enumerate}[label=(D\arabic*)]
\item\label{d1}
$x\in\mathring{F}_x$ and $F_x\subseteq V_x$ for all $x\in X$;
\item\label{d2}
$F_x\cap V_y=\emptyset$ for all $x\ne y$ in $X$;
\item\label{d3}
$F=\cup_{x\in X}F_x$ is closed and recurrent;
\item\label{d4}
for all $x\in X$ and $y\in F_x$,
\begin{equation*}
  \frac{1}{C} < \frac{d\ve(y,V_x)}{d\ve(x,V_x)} < C.
\end{equation*}
\end{enumerate}
We say that $X$ is \emph{$*$-recurrent} if it admits LS-data.
Our requirements  \ref{d1} and \ref{d2} are adopted from \cite{BL2,BP}
and more restrictive than the corresponding ones in \cite{LS}.

Suppose now that we are given regular LS-data as above.
For a finite measure $\mu$ on $M$, define measures
\begin{align}\label{lsmm}
	\mu' = \sum_{x\in X} \int_{F_x}\beta_\mu^F(dy)(\ve_y^{V_x}-\frac1C\ve_x^{V_x})
	\quad\text{and}\quad
	\mu'' = \frac1C\sum_{x\in X} \int_{F_x}\beta_\mu^F(dy)\delta_x
\end{align}
on $M$ with support on $\cup_{x\in X}\partial V_x$ and $X$, respectively.

\begin{prop}[Proposition 3.8 in \cite{BP}]\label{lsh1}
If $h$ is a positive $L$-harmonic function on $M$ swept by $F$ and $\mu$ is a finite measure on $M$, then
\begin{align*}\mu(h)=\mu'(h)+\mu''(h) \quad\text{and}\quad \mu'(h)\le(1-\frac1{C^2})\mu(h).\end{align*}
\end{prop}

For $y\in M$, let now
\begin{equation}\label{mm2}
	\mu_{y,0}=
	\begin{cases}
	\delta_y &\text{if $y\notin X$,} \\
	\ve(y,V_y) &\text{if $y\in X$,} 
\end{cases}
\end{equation}
and set recursively, for $n\ge1$,
\begin{equation}\label{mmn}
  \mu_{y,n} = (\mu_{y,n-1})' \quad\text{and}\quad \tau_{y,n} = (\mu_{y,n-1})''.
\end{equation}
The associated \emph{LS-measure} is the probability measure
\begin{equation}\label{lsm}
  \mu_y = \sum_{n\ge1} \tau_{y,n}
\end{equation}
with support on $X$.

\begin{prop}[Proposition 3.12 in \cite{BP}]\label{lsm2}
For regular LS-data, the associated family $(\mu_y)_{y\in M}$ of LS-measures
has the following properties:
\begin{enumerate}
\item\label{lsm2a}
$\mu_y$ is a probability measure on $X$ such that $\mu_y(x)>0$ for all $x\in X$;
\item\label{lsm2b}
for any $x\in X$ and diffeomorphism $\gamma$ of $M$ leaving $L$, $X$, and the LS-data invariant,
\begin{equation*}
  \mu_{\gamma y}(\gamma x) = \mu_y(x);
\end{equation*}
\item\label{lsm2c}
for all $x\in X$,
\begin{equation*}
  \mu_x = \int_{\partial V_x} \ve_x^{V_x}(dy)\mu_y;
\end{equation*}
\item\label{lsm2d}
for all $x\in X$ and $y\in F_x$ different from $x$,
\begin{equation*}
  \mu_y =  \frac1C\delta_x + \int_{\partial V_x}\ve_x^{V_x}(dz) (\frac{d\ve(y,V_x)}{d\ve(x,V_x)}-\frac1C)\mu_z;
\end{equation*}
\item\label{lsm2e}
for any $y\in M\setminus F$ and stopping time $T\le R^F$,
\begin{equation*}
  \mu_y = \int \pi_y^T(dz)\mu_z,
\end{equation*}
where $\pi_y^T$ denotes the distribution of $P_y$ at time $T$.
\end{enumerate}
\end{prop}

\begin{cor}\label{lsm3}
Let $(\mu_y)_{y\in M}$ be the family of LS-measures associated to regular LS-data.
Assume in addition that the $F_x$, $x\in X$, are compact domains with smooth boundary,
and let $z\in X$.
Then the function $\mu(z)\colon M\to(0,1)$, $y\mapsto\mu_y(z)$, has the following properties:
\begin{enumerate}
\item\label{lsm3a}
For any $x\in X$, we have $\mu(z)=h_x$ on $F_x\setminus\{x\}$,
where $h_x$ is the $L$-harmonic function on $V_x$ given by
\begin{align*}
	h_x(y) = \ve_y^{V_x}(\mu(z)) + c_x
\end{align*}
with $c_x=(\delta_x(z)-\mu_x(z))/C$.
Moreover, $\mu(z)$ is discontinuous at any $x\in X$,
\begin{align*}
	\mu_x(z) = \ve_x^{V_x}(\mu(z)).
\end{align*}
\item\label{lsm3b}
The restriction of $\mu(z)$ to $M\setminus F$ is $L$-harmonic
and solves the Dirichlet problem $\mu(z)=h_x$ on $\partial F_x$, for all $x\in X$.
In particular, $\mu(z)$ is continuous on $M\setminus X$.
\end{enumerate}
\end{cor}

\begin{proof}
\eqref{lsm3a} amounts to a translation of \cref{lsm2}.\ref{lsm2c} and \ref{lsm2}.\ref{lsm2d}.
The first claim of \eqref{lsm3b} follows from \cref{lsm2}.\ref{lsm2e} by choosing $T=R^F$.
As for the Dirichlet problem,
we may choose $T=R^F\wedge S^V$ on $V_x\setminus F_x$ in \cref{lsm2}.\ref{lsm2e}.
Since the boundary $\partial F_x$ of the domain $F_x$ is smooth,
the distribution of $P_y$ at time $T$ tends to the Dirac measure at $y_\infty\in\partial F_x$
as $y\in V_x\setminus F_x$ tends to $y_\infty$.
\end{proof}

The requirement on the smoothness of the $\partial F_x$ in \cref{lsm3} can be weakened.
We only use it to guarantee that the distribution of $P_y$ at the random time $T$ (as above)
tends to the Dirac measure at $y_\infty\in\partial F_x$
as $y\in V_x\setminus F_x$ tends to $y_\infty$.

\subsection{Exponential moments of LS-measures}
\label{susemom}
Although the following could be discussed in greater generality,
we now come back to one of the setups in the introduction
and let $L$ be a diffussion operator on a manifold $M$
which is invariat under a group $\Gamma$ acting properly discontinuously and cocompactly on $M$.
Clearly, the Riemannian metric assocated to $L$ is also invariant under $\Gamma$.
In particular, $M$ is complete with respect to the associated distance $d$.

For $x\in M$,
we denote by $B(x,r)$ and $\bar B(x,r)$ the open and closed ball of radius $r$ about $x\in M$
with respect to $d$
and call
\begin{align*}
	D_x = \{ y\in M \mid \text{$d(y,x) \le d(y,gx)$ for all $g\in\Gamma$}\}
\end{align*}
the \emph{Dirichlet domain of $x$ with respect to $\Gamma$.}

We choose an origin $x_0\in M$ and set $X=\Gamma x_0$ and $D_0=D_{x_0}$.
We let $V_0=V_{x_0}$ be a relatively compact and connected domain with smooth boundary
such that $V_0$ is invariant under the isotropy group $\Gamma_0$ of $x_0$
and such that, for some $\ve>0$,
\begin{align*}
	  B(x_0,\ve) \subseteq V_0
	  \quad\text{and}\quad
	  V_0 \cap B(x,\ve) = \emptyset
\end{align*}
for all $x\in X$ with $x\ne x_0$.
For convenience, we also require that
\begin{align*}
	B(x_0,\ve) \subseteq D_0
\end{align*}
and choose a $\Gamma_0$-invariant compact domain
\begin{align*}
	F_0 = F_{x_0} \subseteq B(x_0,\ve)
\end{align*}
with smooth boundary.
For each $x\in X$, we now set
\begin{align*}
	F_x = gF_0 \quad\text{and}\quad V_x = gV_0,
\end{align*}
where $x=gx_0$ with $g\in\Gamma$.
Since $F_0$ and $V_0$ are invariant under $\Gamma_0$,
$F_x$ and $V_x$ are well-defined.
By the choices of $F_0\subseteq V_0$ and $\ve>0$, we have
\begin{align*}
  F_x \subseteq V_x
  \quad\text{and}\quad
  F_x \cap V_y = \emptyset
\end{align*}
for all $x,y\in X$ with $x\ne y$.
Since the action of $\Gamma$ is properly discontinuous and cocompact,
the family of $(F_x,V_x)_{x\in X}$ are regular LS-data in the sense of \cref{sublsdh}.
We denote the corresponding Harnack constant of the pairs $(F_x,V_x)$ by $C$,
and let $(\mu_y)_{y\in M}$ be the family of LS-measures on $X$ associated to the data.
Since the data are invariant under $\Gamma$, we have
\begin{align*}
	\mu_{\gamma y}(\gamma x) = \mu_y(x)
\end{align*}
for all $\gamma\in\Gamma$, $x\in X$, and $y\in M$.

\begin{lem}\label{mubound}
There is a constant $C_D$ such that $\mu_y(z)\le C_D\mu_{x}(z)$ for any $x\in X$, $y\in D_x$, and $z\in X$.
\end{lem}

\begin{proof}
Let $C>1$ be the Harnack constant as in \ref{d4}.
From \cref{lsm2}.\ref{lsm2d}, we get that, for any $x\in X$ and $y \in F_x \setminus \{x\}$,
\begin{align*}
	\mu_{y}(z)
	&\le \frac{1}{C} \delta_x(z) + (C - \frac{1}{C}) \int_{\partial V_x} \ve_x^{V_x}(du) \mu_{u}(z)  \\
	&\le \frac{1}{C} \delta_x(z) + (C - \frac{1}{C}) \mu_x(z).
\end{align*}
Therefore $\mu_y(z)\le C\mu_x(z)$ for all $y\in F_x$ and $z\in X\setminus\{x\}$.
Since $\mu_x(x)>0$, there is also a constant $C'>0$ such that $1/C\le C'\mu_x(x)$,
and then $\mu_y(z) \le (C+C')\mu_x(z)$ for all $y\in F_x$ and $z\in X$.

Fix an open domain $U_x$ with smooth boundary such that
\begin{align*}
	F_x \subseteq U_x \subseteq \bar U_x \subseteq B(x,\ve),
\end{align*}
and let $C''>1$ be the Harnack constant for the pair $(\bar U_x , V_x)$ with respect to $L$.
For $y\in\bar U_x\setminus F_x$, denote by $\ve_y$ the exit measure from $V_x\setminus F_x$.
Then
\begin{align*}
	\ve_y|_{\partial V_x} \le \ve_y^{V_x} \le C''\ve_x^{V_x}.
\end{align*}
From \cref{lsm3}.\ref{lsm3b} and the first part of the proof, we get
\begin{align*}
	\mu_y(z) &= \ve_y(\mu(z)) \\
	&= \int_{\partial F_x} \ve_y(du)\mu_u(z) +  \int_{\partial V_x} \ve_y(du)\mu_u(z) \\
	&\le (C+C')\mu_x(z) \ve_y(\partial F_x) + C'' \int_{\partial V_x} \ve_x^{V_x}(du)\mu_u(z) \\
	&\le (C+C'+C'') \mu_x(z).
\end{align*}
Now there is a constant $r>0$ such that $d(y,F_x)\ge r$ for any $y\in\partial U_x$.
Hence we may apply the Harnack inequality of Cheng-Yau \cite[Theorem 6]{CY}
to $\mu(z)$ in pairs of balls of radius $r/2$ and $r$ along minimal paths
connecting a point $y\in D_x\setminus U_x$ to $\partial U_x$ consecutively
to arrive at the desired estimate for any given $x\in X$.
However, $\Gamma$-invariance implies that the same estimate holds for all $x\in X$.
\end{proof}

We let $X_0=\{x_0\}$ and $U_0$ be a relatively compact open, connected,
and $\Gamma_0$-invariant  neighborhood of $D_0$
such that $U_0\cap B(x,\ve)=\emptyset$ for all $x\in X$ with $x\ne x_0$.
For $x=gx_0\in X$, we let $U_x=gU_0$.
By recursion, we set
\begin{align}\label{xnun}
  X_n = \{x\in X \mid U_x \cap U_{n-1} \ne \emptyset\}
  \quad\text{and}\quad
  U_n = \cup_{x \in X_n} U_x.
\end{align}
Then
\begin{align}\label{un}
	U_0 \subseteq U_1 \subseteq U_2 \subseteq \cdots
\end{align}
is an exhaustion of $M$ by relatively compact open subsets such that
\begin{align}\label{un2}
	\bar U_n \subseteq U_{n+1}
	\quad\text{and}\quad
	\partial U_n \cap F = \emptyset
\end{align}
for all $n\ge0$.
Furthermore,
\begin{align}\label{un4}
	U_n \subseteq B(x_0,(n+1)\diam U_0)
\end{align}
for all $n\ge0$.
Finally, we fix a constant $0<c_0<1$ such that
\begin{align}\label{c0}
  \ve_{z}^{U_0\setminus F_0}(F_0) \ge c_0
  \quad\text{for any $y\in D_0$.}
\end{align}

\begin{lem}\label{subexp}
For any $n\ge0$ and $y\in D_0$,
we have $\ve_{y}^{U_n\setminus F}(\partial U_n)\le(1-c_0)^{n+1}$.
\end{lem}

\begin{proof}
By the definition of $c_0$, the assertion holds for $n=0$ (and any $y\in D_0$).
Assume now that it holds for some $n\ge0$.
Given $y\in D_0$,
the strong Markov property of the $L$-process together with $\partial U_n\cap F=\emptyset$ yields that
\begin{equation}\label{smp}
	\ve_{y}^{U_{n+1}\setminus F}(\partial U_{n+1})
	= \int_{\partial U_n} \ve_{y}^{U_n \setminus F}(dz) \ve_{z}^{U_{n+1} \setminus F}(\partial U_{n+1}).
\end{equation}
For any $z\in\partial U_n$,
there exists $u\in X_{n+1}$ such that $z\in D_u$, by \eqref{xnun}.
Clearly
\begin{align*}
	\ve_{z}^{U_{n+1}\setminus F}(\partial U_{n+1})
	\le \ve_{z}^{U_u\setminus F}(\partial U_u)
	\le 1 - c_0,
\end{align*}
where we use the $\Gamma$-equivariance of the data in the second step.
Therefore 
\begin{align*}
	\ve_{y}^{U_{n+1}\setminus F}(\partial U_{n+1})
	\le (1 - c_0) \ve_{y}^{U_n\setminus F}(\partial U_n) \leq (1 - c_0)^{n+2},
\end{align*}
by \eqref{smp}.
This completes the inductive step.
\end{proof}

\begin{thm}\label{exmom}
With LS-data as above,
the associated LS-measures have finite exponential moments.
More precisely,
\begin{align*}
\sum_{x\in X}\mu_y(x)e^{\alpha d(x,y)}<\infty
\end{align*}
for all $y\in M$ and $\alpha>0$ sufficiently small.
\end{thm}

\begin{cor}\label{amom}
With LS-data as above,
the associated LS-measures have finite $a$-moments for any subexponential growth function $a$.
\end{cor}

\begin{rem}\label{rembg}
It is important in our arguments that we use a refined version of the LS-construction
which goes back to \cite{BL2} in the case of Brownian motion
and was discussed for diffusion operators in \cite{BP}.
The proof that the LS-measures in the original construction of Lyons and Sullivan
have exponential moments in the cocompact case in \cite[Lemma 3.13]{Ba}
does not apply immediately in the present situation.
We owe the main argument here to Fran\c{c}ois Ledrappier.
\end{rem}

Before starting with the proof of \cref{exmom}, we introduce some further notation.
For $\omega\in\Omega$, let
\begin{equation}\label{s0}
  S_0(\omega)= 
  \begin{cases}
  0 &\text{if $\omega(0)\notin X$}, \\
  S^{V_x}(\omega) &\text{if $\omega(0)=x\in X$,}
\end{cases}
\end{equation}
and recursively, for $n\ge1$,
\begin{equation}\label{rs}
\begin{split}
  R_n(\omega) &= \inf \{ t\ge S_{n-1}(\omega) \mid \omega(t)\in F \}, \\
  S_n(\omega) &= \inf \{ t\ge R_n(\omega) \mid \omega(t)\notin V_{x_n(\omega)} \},
\end{split}
\end{equation}
where $x_n=x_n(\omega)\in X$ with $y_n=y_n(\omega)=\omega(R_n(\omega))\in F_{x_n(\omega)}$.

\begin{proof}[Proof of \cref{exmom}]
We may assume that $y\in D_0$.
Using \cref{lsm2}.\ref{lsm2c}, we may also assume that $y\ne x_0$.
We now set $\mu_{y,0}=\delta_y$ and $\mu_{y,n}=\mu_{y,n-1}'$ as in \eqref{mm2} and \eqref{mmn}
and get
\begin{align*}
  \mu_{y,n} = \mu_{y,n-1}'
  &= \sum_{x\in X} \int_{F_x}\beta_{\mu_{y,n-1}}^F(du)(\ve_u^{V_x} - \frac1C\ve_x^{V_x}) \\
  &\le \theta \sum_{x\in X} \int_{F_x}\beta_{\mu_{y,n-1}}^F(du)\ve_u^{V_x}
  = \theta \pi_{\mu_{y,n-1}}^{S_1},
\end{align*}
where $\theta=1-C^{-2}$
and $\pi_\mu^S$ denotes the distribution of $P_\mu$ at the random time $S$.
Since $\mu_{y,n-1}=(\mu_{y,n-2})'$, we can proceed by recursion and get
\begin{align*}
  \mu_{y,n}
  \le \theta \pi_{\mu_{y,n-1}}^{S_1}
  \le \theta^2 \pi_{\mu_{y,n-2}}^{S_2}
  \le \dots \le \theta^n \pi_{\delta_y}^{S_n}.
\end{align*}
Now $R_n$ is the first time of hitting $F$ after $S_{n-1}$, and hence we also get
\begin{align*}
  \mu_{y,n}''(x) = \frac1C \mu_{y,n}(\partial V_x)
  \le \frac1C \theta^n \pi_{\delta_y}^{S_n}(\partial V_x)
  = \frac1C \theta^n P_y[R_n(\omega)\in F_x].
\end{align*}
Therefore
\begin{align}\label{est0}
\begin{split}
	\sum _{x\in X} \mu_y(x)e^{\alpha d(x,y)}
	&\le \frac1C \sum_{x\in X} \sum_{n\ge1} \theta ^n P_y[R_n(\omega )\in F_x] e^{\alpha d(x,y)} \\
	&= \frac1C \sum_{n\ge1}\theta ^n E_y[e^{\alpha d(x_n(\omega),y)}] \\
	&\le \frac{e^{\alpha\ve}}C \sum_{n\ge1}\theta ^n E_y[e^{\alpha d(y_n(\omega),y)}],
\end{split}
\end{align}
where we use that $y_n(\omega)\in F_{x_n(\omega)}\subseteq B(x_n(\omega),\ve)$.
By the Markov property of the $L$-process and the triangle inequality, 
\begin{align*}
	E_y[e^{\alpha d(y_n(\omega),y)}] \le \sup_{u\in D_0} \left( E_u[e^{\alpha d(y_1(\omega),u)}] \right)^n.
\end{align*}
By \eqref{un4} and \cref{subexp},
\begin{align}\label{est1}
	P_u[d(y_1(\omega),u) \ge n\diam U_0] \le \ve_u^{U_{n-2} \setminus F}(\partial U_{n-2}) \le (1-c_0)^{n-1}
\end{align}
for any $u\in D_0$.
Hence
\begin{align*}
	E_u[e^{\alpha d(y_1(\omega),u)}]
	= \int_\Omega P_u(d\omega) e^{\alpha d(y_1(\omega),u)}
	\le \sum_{n} (1-c_0)^{n-1}e^{\alpha(n+1)\diam U_0}.
\end{align*}
The sum on the right is finite for $(1-c_0)e^{\alpha\diam U_0}<1$,
hence the integral on the left is finite.
From \eqref{est1}, we also get that then
\begin{align}\label{est2}
\begin{split}
	\int_{d(y_1(\omega),u)>k\diam D_0} P_u(d\omega) e^{\alpha d(y_1(\omega),u)}
	&\le \sum_{n\ge k} (1-c_0)^{n-1}e^{\alpha(n+1)\diam U_0} \\
	&= (1 - c_0)^{-2} \frac{\{(1-c_0)e^{\alpha\diam U_0}\}^{k+1}}
	{1-(1-c_0)e^{\alpha\diam U_0}},
\end{split}
\end{align}
which tends to zero for $k\to\infty$ and uniformly for small $\alpha\ge0$.
The integral over the part of $\Omega$, where $d(y_1(\omega),u)\le k\diam U_0$,
is bounded by $e^{\alpha k\diam U_0}$,
which tends to $1$ as $\alpha\to0$.
Hence we may choose $\alpha>0$ and $k\ge1$ such that
\begin{align}\label{est4}
	1 <
	\left(1 + (1-c_0)^{k-2}\frac{(1-c_0)e^{\alpha\diam U_0}}
	{1-(1-c_0)e^{\alpha\diam U_0}} \right)e^{\alpha k\diam U_0}
	< \theta^{-1}
	= \frac{C^2}{C^2-1}.
\end{align}
Then $\sup_{u\in D_0} E_u[e^{\alpha d(y_n(\omega),u)}]<\theta^{-1}$,
and the right hand side of \eqref{est0} is finite.
\end{proof}

\subsection{Balanced LS-data}
\label{susebal}
In the situation considered in \cref{susemom},
we let $G_0(.,.)$ be the $L$-Green function of $V_0$.
Since $G(y,x_0)\to\infty$ as $y\to x_0$,
we can choose a constant $B$ such that $B$ is a regular value of $G(.,x_0)$ and let $F_0 = F_{x_0}$ be the connected component of
\begin{align*}
 \{G_0(.,x_0)\ge B\} \subseteq B(x_0,\ve)
\end{align*}
containing $x_0$. Since $V_0$ is invariant under $\Gamma_0$,
$G(.,x_0)$ is invariant under $\Gamma_0$ as well, hence also $F_0$.
Now we proceed as in \cref{susemom} to get regular LS-data $(F_x,V_x)$.
They are \emph{balanced} in the sense of \cite{BL2,BP}.
The following is \cite[Theorem 2.7]{BL2} in the case of Brownian motion.

\begin{thm}[Theorem 3.29 in \cite{BP}]\label{thmbal}
Let $(\mu_y)_{y\in M}$ be the family of LS-measures on $X$ associated to balanced LS-data as above.
Then the $F_x$, $x\in X$, are compact domains with smooth boundary, and we have:
\begin{enumerate}
\item\label{thmbal1}
The Green functions $G$ of $L$ on $M$ and $g$ of the random walk on $X$ asssociated to the family $(\mu_y)_{y\in M}$ of LS-measures satisfy
\begin{align*}
  G(y,x) = BC g(y,x) \quad\text{for all $x\in X$ and $y\in M \setminus V_x$.}
\end{align*}
\item\label{thmbal2}
The $L$-diffusion on $M$ is transient if and only if the random walk on $X$ asssociated to the family $(\mu_y)_{y\in M}$ of LS-measures is transient,
and then $\mu_y(x)=\mu_x(y)$ for all $x,y\in X$.
\end{enumerate}
\end{thm}

\subsection{Associated random walk on $\Gamma$}
\label{suseran}
In the situation considered in \cref{susemom} and \cref{susebal},
we may choose $x_0\in M$ with trivial isotropy group, $\Gamma_0=\{1\}$.
Then we may identify $\Gamma$ via the orbit map $g\mapsto gx_0$ with $X=\Gamma x_0$.
Under this identification, $\mu_{x_0}$ induces a probability measure $\mu$ on $\Gamma$
by $\mu(\gamma)=\mu_{x_0}(\gamma x_0)$.

\begin{prop}\label{court}
The probability measure $\mu$ has the following properties:
\begin{enumerate}
\item\label{courtsu}
$\mu(\gamma)>0$ for all $\gamma\in\Gamma$.
\item\label{courtsy}
If the $\mu$-random walk on $\Gamma$ is transient,
then $\mu(\gamma^{-1})=\mu(\gamma)$ for all $\gamma\in\Gamma$.
\item\label{courtem}
$\sum_{\gamma\in\Gamma}\mu(\gamma)e^{\alpha|\gamma|}<\infty$
for all sufficiently small $\alpha>0$.
\end{enumerate}
\end{prop}

\begin{proof}
\eqref{courtsu} follows immediately from \cref{lsm2}.\ref{lsm2a},
\eqref{courtem} from \cref{exmom}.
As for \eqref{courtsy}, the Lyons-Sullivan measures on $X=\Gamma x_0$ satisfy
\begin{align*}
  \mu_{x_0}(\gamma^{-1} x_0) = \mu_{\gamma x_0}(x_0) = \mu_{x_0}(\gamma x_0)
\end{align*}
for all $\gamma\in\Gamma$,
by \cref{lsm2}.\ref{lsm2b} and \cref{thmbal}.\ref{thmbal2}.
Now \eqref{courtsy} follows immediately from the definition of $\mu$.
\end{proof}

Note that $\mu$ satisfies the properties required in \cite{MPTY,MY}
in the case where the $\mu$-random walk is transient.

\section{Cocompact actions}
\label{secca}

The purpose of this section is to prove a general version of \cref{thmbg}
in the case where $L$ is invariant under a group $\Gamma$
which acts properly discontinuously and cocompactly on $M$.
We fix an origin $x_0\in M$ and let $X=\Gamma x_0$.
We also choose $\Gamma$-invariant regular LS-data as in \cref{seclsd}
and let $\mu=(\mu_y)_{y\in M}$ be the associated LS-measures on $X$.
Finally, we fix a growth function $a$.

\begin{thm}\label{thmba}
If $\mu$ has finite $a$-moments,
then the restriction of an $a$-bounded $L$-harmonic function on $M$ to $X$
is $a$-bounded and $\mu$-harmonic,
and the restriction map $\H_a(M,L)\to\H_a(X,\mu)$ is a $\Gamma$-equivariant isomorphism.
\end{thm}

\begin{proof}
We begin by showing that $a$-bounded $\mu$-harmonic functions on $X$
extend to $a$-bounded $L$-harmonic functions on $M$.
To this end, we let $h\in\H_a(X,\mu)$ and define
\begin{align}\label{exten}
	f\colon M\to\R, \quad f(y) = \mu_{y}(h) = \sum_{x \in X} \mu_{y}(x)h(x).
\end{align}
First of all, we note that $f$ is well-defined since
\begin{align*}
  \sum_{x \in X} \mu_{y}(x) |h(x)|
  \le C_h \sum_{x \in X} \mu_{y}(x) a(|x|) < \infty.
\end{align*}

\begin{lem}\label{uniform}
With $X_n$ as in \eqref{xnun}, let
\begin{align*}
  f_n(y) = \sum_{x \in X_{n}} \mu_{y}(x)h(x).
\end{align*}
Then the sequence of functions $f_n$ converges locally uniformly to the function $f$.
\end{lem}

\begin{proof}
Let $C_D$ be the constant from \cref{mubound}.
Since $M$ is covered by the Dirichlet domains $D_x$, $x\in X$,
it suffices to consider the compact sets $D_x$, $x\in X$.
Let $\ve>0$, fix $x \in X$, and choose $n_0\in\N$ such that
\begin{align*}
	\sum_{u \in X\setminus X_n} \mu_{x}(u)a(|u|) < \ve/C_hC_D
\end{align*}
for all $n \geq n_{0}$.
Then we have, for any $y\in D_x$ and $n\ge n_0$,
\begin{align*}
  \left| f(y) - \sum_{u \in X_{n}} \mu_{y}(u) h(u) \right|
  &\le C_h \sum_{u \in X \setminus X_{n}} \mu_{y}(u) a(|u|) \\
  &\le C_hC_D \sum_{u \in X \setminus X_{n}} \mu_{x}(u) a(|u|)
  < \ve.
\end{align*}
This shows that the sequence of functions $f_n$ converges uniformly to $f$ on $D_x$ for any $x \in X$.
\end{proof}

\begin{lem}\label{lemlh}
The function $f$ is $L$-harmonic.
\end{lem}

\begin{proof}
By \cref{lsm3}.\ref{lsm3b},
the functions $\mu(u)$ are $L$-harmonic on $M\setminus F$.
Hence the functions $f_n$ as in \cref{uniform} are $L$-harmonic on $M\setminus F$.
Therefore the limit function $f$ is also $L$-harmonic on $M\setminus F$.

It suffices now to prove that $f$ is $L$-harmonic in $V_x$, for any $x \in X$.
Consider first a point $y \in F_x \setminus \{x\}$ and let $u\in X$.
Then we have that
\begin{align*}
	\mu_{y}(u)
	= \frac{1}{C} \delta_x(u) + \int_{\partial V_x} \ve_{y}^{V_x}(dz) \mu_{z}(u)  - \frac{1}{C} \mu_x(u).
\end{align*}
Therefore, by the uniform convergence $f_n\to f$ on $\bar V_x$,
\begin{align*}
  \int_{\partial V_x} \ve_{y}^{V_x}(dz) f(z) 
  &= \lim_{n\to\infty} \sum_{u \in X_{n}} \left( \int_{\partial V_x} \ve_{y}^{V_x}(dz) \mu_{z}(u)  \right) h(u) \\
  &= \lim_{n\to\infty} \sum_{u \in X_{n}} (\mu_{y}(u) - \frac{1}{C} \delta_x(u) + \frac{1}{C} \mu_x(u))h(u) \\
  &= f(y) - \frac{1}{C} h(x) + \frac{1}{C} \sum_{u \in X} \mu_x(u)h(u)
  = f(y),
\end{align*}
where we use that $h$ is $\mu$-harmonic. Similarly, it follows that
\begin{align*}
  \int_{\partial V_x} \ve_x^{V_x}(dz) f(z)  = f(x).
\end{align*}
Hence $f(y)=\ve_y^{V_x}(f)$ for all $y\in F_x$.

Let now $y\in V_x\setminus F_x$, and denote by $\ve_y$ the exit measure from $V_x\setminus F_x$.
Then
\begin{align*}
	\ve_y^{V_x} = \ve_y|_{\partial V_x} + \int_{\partial F_x} \ve_y(dz) \ve_z^{V_x},
\end{align*}
where the first term on the right corresponds to the paths which leave $V_x$ before entering $F_x$
and the second term to the paths which enter $F_x$ before leaving $V_x$.
By the uniform convergence $f_n\to f$ on $V_x$, $f$ is continuous on $V_x\setminus\{x\}$.
Moreover, by the first part of the proof, $f$ is $L$-harmonic on $V_x\setminus F_x$
and satisfies the mean value formula on $F_x$.
Hence
\begin{align*}
	f(y) = \ve_y(f)
	&= \int_{\partial V_x} \ve_y(dz) f(z) + \int_{\partial F_x} \ve_y(dz) f(z) \\
	&= \int_{\partial V_x} \ve_y(dz) f(z) + \int_{\partial F_x} \ve_y(dz) \ve_z^{V_x}(f)
	= \ve_y^{V_x}(f).
\end{align*}
We conclude that $f$ satisfies the mean value formula also on $V_x\setminus F_x$,
and hence $f$ is $L$-harmonic on $V_x$.
\end{proof}

\begin{lem}\label{lemab}
The function $f$ is $a$-bounded.
\end{lem}

\begin{proof}
Let $D$ be the Dirichlet domain of $x_0$ with respect to $\Gamma$.
Let $z\in M$ and write $z=gy$ with $g\in \Gamma$ and $y\in D$.
Using the triangle inequality and the monotonicity and sub\-multiplicativity of $a$, we get
\begin{align*}
	|f(z)| &\le \sum_{x \in X} \mu_{z}(x) |h(x)|
	\le C_h \sum_{x \in X} \mu_{z}(x) a(|x|) \\
	&= C_h \sum_{x \in X} \mu_{g^{-1}z}(g^{-1}x) a(|x|)
	= C_h \sum_{x \in X} \mu_y(x) a(|gx|) \\
	&\le C_h\sum_{x \in X} \mu_y(x)a(|gx_0|+|x|)
	\le C_aC_h a(|gx_0|)\sum_{x \in X} \mu_y(x)a(|x|) \\
	&\le C_aC_hC_D a(|gx_0|)\sum_{x \in X} \mu_{x_0}(x)a(|x|)
	= C_aC_hC_D c_0 a(|gx_0|) \\
	&\le \{C_a^2C_hC_Dc_0a(\diam D)\} a(|z|),
\end{align*}
where $C_a$ is the constant from \eqref{submul},
$C_D$ the constant from \cref{mubound},
and $c_0$ the $a$-moment of $\mu_{x_0}$.
\end{proof}

Lemmata \ref{lemlh} and \ref{lemab} show that the extension $f$ of an $a$-bounded $\mu$-harmonic function
$h$ on $X$ as in \eqref{exten} is an $a$-bounded $L$-harmonic function on $M$.
To finish the proof of \cref{thmba},
it remains to show that the restriction of any $a$-bounded $L$-harmonic function $f$ on $M$
to $X$ is $\mu$-harmonic.

\begin{lem}\label{lemsweep}
Any function $f\in\H_a(M,L)$ is swept by $F$, that is,
\begin{align*}
	f(y) = \beta_y^{F}(f) \quad\text{for any $y \in M$.}
\end{align*}
\end{lem}

\begin{proof}
The assertion is obvious for $y\in F$.
Let now $y\in M\setminus F$, and choose $g\in\Gamma$ with $y\in gD$, where $D=D_{x_0}$.

First of all, note that $\beta_y^{F}(f)$ is finite. Indeed, for $x \in X$ and $z \in \partial F_x$,
we have that $|z| \le |x| + \diam(F_{x_0})$.
Therefore
\begin{align*}
	|\beta_y^{F}(f)|
	&\le \sum_{x\in X} \int_{F_x} \beta_y^{F}(dz) |f(z)|
	\le C_f \sum_{x \in X} \int_{F_x} \beta_y^{F}(dz) a(|z|) \\
	&\le C_f C_a a(\diam(F_{x_0})) \sum_{x \in X} \beta_y^{F}(F_x) a(|x|) \\
	&\le C_f C_a C a(\diam(F_{x_0})) \sum_{x \in X} \mu_y(x) a(|x|) < \infty,
\end{align*}
where we used that $\beta_y^{F}(F_x)/C\le\tau_{y,1}(x)\le\sum_{n\ge1}\tau_{y,n}(x)=\mu_y(x)$.

Consider the exhausting sequence of $gU_n$ of $M$ by the relatively compact open subsets defined in \eqref{xnun}.
Since $f$ is harmonic, we have
\begin{equation}\label{terms}
	f(y)
	= \sum_{x\in gX_n} \int_{\partial F_x} \ve_y^{gU_n\setminus F}(dz) f(z) 
	+ \int_{\partial (gU_n)} \ve_y^{gU_n\setminus F}(dz) f(z).
\end{equation}
Observe that the last term converges to zero as $n \rightarrow \infty$.
Indeed, we have that
\begin{align*}
	\bigg| \int_{\partial (gU_n)} \ve_y^{gU_n \setminus F}(dz) f(z) \bigg|
	&\le \int_{\partial (gU_n)} \ve_y^{gU_n \setminus F}(dz) |f(z)| \\
	&\le C_f \int_{\partial (gU_n)} \ve_y^{gU_n \setminus F}(dz) a(|z|) \\
	&\le C_f \sum_{x\in g(X_{n+1}\setminus X_n)}
		\int_{\partial (gU_n)\cap D_x} \ve_y^{gU_n \setminus F}(dz) a(|z|) \\
	&\le C_fC_1 \sum_{x\in g(X_{n+1}\setminus X_n)}
		\int_{\partial (gU_n)\cap D_x} \ve_y^{gU_n \setminus F}(dz) a(|x|) \\
	&\hspace{-15mm}\le C_f\frac{C_1}{C_0} \sum_{x\in g(X_{n+1}\setminus X_n)}
		\int_{\partial (gU_n)\cap D_x} \ve_y^{gU_n \setminus F}(dz)
		\ve_z^{U_x\setminus F_x}(F_x) a(|x|) \\
	&\le C_f\frac{C_1}{C_0} \sum_{x\in g(X_{n+1}\setminus X_n)} \beta_y^{F}(F_x) a(|x|) \\
	&\le C_f\frac{C_1}{C_0}C \sum_{x\in g(X_{n+1}\setminus X_n)} \mu_y(x) a(|x|),
\end{align*} 
where $C_0$ is a constant satisfying $\ve_z^{U_x\setminus F_x}(F_x)\ge C_0$
for all $x\in X$ and $z\in D_x\setminus F_x$, $C_1=C_a a(\diam D_x)$,
and $C$ is the Harnack constant from \ref{d4} as used above.
Now the last term tends to $0$ as $n\to\infty$ since $\mu_y$ has finite $a$-moments.

It remains to prove that the first term in (\ref{terms}) converges to $\beta_y^{F}(f)$. Let $\ve > 0$ and note that there exists $n_0 \in \mathbb{N}$ such that
\begin{align*}
\sum_{x \in X \setminus g X_{n_0}} \int_{\partial F_x} \beta_y^{F}(dz) |f(z)|  < \ve/3.
\end{align*}
For any $n \in \mathbb{N}$ and $x \in g X_{n_0}$ we have that $\ve_y^{g U_n \setminus F}(A) \leq \beta_y^{F}(A)$ for any Borel subset $A$ of $\partial F_x$. This yields that $\beta_y^{F} - \ve_y^{g U_n \setminus F}$ is a measure on the Borel subsets of $\partial F_x$. Moreover, we have that $\ve_y^{g U_n \setminus F}(\partial F_x) \rightarrow \beta_y^{F}(\partial F)$, which implies that for any $x \in g X_{n_0}$ there exists $n_x \in \mathbb{N}$ such that
\begin{align*}
\beta_y^{F}(\partial F_x) - \ve_y^{g U_n \setminus F}(\partial F_x) < \frac{\ve}{3 C_f C_a a(\diam (F_{x_0})) | X_{n_0}| a(|x|) }
\end{align*}
for any $n \geq n_x$.

Then, for $n \geq \max\{n_0 ,\max_{x \in X_{n_0}} n_x\}$, we derive that
\allowdisplaybreaks[2]
\begin{align*}
	\bigg| \beta_y^{F}(f) &- \sum_{x \in gX_n} \int_{\partial F_x} \ve_y^{gU_n \setminus F}(dz)  f(z) \bigg| \\
	&\le \bigg| \sum_{x \in gX_{n_0}} \int_{\partial F_x} (\beta_y^{F} - \ve_y^{gU_n \setminus F})(dz)  f(z) \bigg|  \\
	&\hspace{44mm}+ \sum_{x \in X \setminus gX_{n_0}} \int_{\partial F_x} (\beta_y^{F} + \ve_y^{gU_n \setminus F})(dz) |f(z)| \\
	&\le  \sum_{x \in gX_{n_0}} \int_{\partial F_x} (\beta_y^{F} - \ve_y^{gU_n \setminus F})(dz)  |f(z)| 
	+ 2 \sum_{x \in X \setminus gX_{n_0}} \int_{\partial F_x} \beta_y^{F}(dz) |f(z)| \\
	&\le \sum_{x \in g X_{n_0}} (\beta_y^{F}(\partial F_x) - \ve_y^{g U_n \setminus F}(\partial F_x)) \sup_{z \in F_x} |f(z)| + 2\ve / 3  \le \ve,
\end{align*}
where we used that $\beta_y^{F}\ge\ve_y^{gU_n\setminus F}$ on the $\partial F_x$ with $x\in gX_{n_0}$.
\end{proof}

\begin{lem}\label{lemmh}
For any function $f\in\H_a(M,L)$, we have
\begin{align*}
	f(y) = \mu_{y}(f) \quad\text{for any $y \in M$.}
\end{align*}
\end{lem}

\begin{proof}
For a finite measure $\mu$ on $M$,
we define the measures $\mu'$ and $\mu''$ as in \eqref{lsmm}.
Since $f$ is swept by $F$,
\begin{align*}
	\mu(f) = \mu'(f) + \mu''(f).
\end{align*}
Observe that $|f(y)|\le C_f a(|x|+\diam(F_{x_{0}}))=:\vf(x)$ for any $x \in X$ and $y \in F_x$. Then we obtain that
\begin{align*}
	|\mu'(f)|
	&= \bigg| \sum_{x \in X} \int_{F_x} \beta_{\mu}^{F}(dy) (f(y) - \frac{1}{C} f(x)) \bigg| \\
	&\le \sum_{x \in X} \int_{F_x} \beta_{\mu}^{F}(dy) (|f(y)| + \frac{1}{C} |f(x)|) \\
	&\le (C + \frac{1}{C}) \sum_{x \in X} \beta_{\mu}^{F}(F_x) \vf(x)
	= (C^{2} + 1) \mu''(\vf).
\end{align*}
In the notation of \cref{sublsdh}, for any $y \in M$, we obtain that
\begin{align*}
	f(y) = \mu_{y,n}(f) + \sum_{1 \leq k \leq n} \tau_{y,k}(f)
	\text{ and }
	|\mu_{y,n}(f)| \leq (C^{2} + 1) \tau_{y,n}(\vf).
\end{align*}
Since the $a$-moments of the Lyons-Sullivan measures are finite,
we have
\begin{align*}
	\sum_{n\ge1} \tau_{y,n}(\vf) = \mu_{y}(\vf) < \infty,
\end{align*}
which yields that $\tau_{y,n}(\vf)\to0$ as $n\to\infty$,
as we wished.
\end{proof}

\cref{lemmh} shows that the restriction to $X$ of an $a$-bounded $L$-harmonic
function on $M$ is $\mu$-harmonic.
Thus the proof of \cref{thmba} is complete.
\end{proof}

\section{Cocompact coverings}
\label{seccov}

Let $q\colon\tilde M\to M$ be a covering of connected manifolds.
Let $\Omega_M$ and $\Omega_{\tilde M}$ be the spaces of continuous paths $\omega$
from $[0,\infty)$ to $M$ and $\tilde M$, respectively.
From the path lifting property of $q$, we obtain a map
\begin{align}\label{omega}
	H \colon \{(x,\omega)\in\tilde M\times\Omega_M\mid q(x)=\omega(0)\} \to \Omega_{\tilde M}, \quad
	H(x,\omega) = \omega_x,	
\end{align}
where $\omega_x$ denotes the continuous lift of $\omega$ to $\tilde M$ starting at $x$.
It is easy to see that $H$ is a homeomorphism with respect to the compact-open topology.
In what follows, we identify $\Omega_{\tilde M}$ according to \eqref{omega}.
With respect to this identification,
evaluation of $(x,\omega)\in\Omega_{\tilde M}$ at time $t\ge0$ is given by $\omega_x(t)$.

Let $L$ be a diffusion operator on $M$ and $\tilde L=q^*L$ be the pull-back of $L$ to $\tilde M$.
Assume that the $L$-diffusion on $M$ is complete,
and denote by $P_y$ the probability measure on $\Omega_M$
corresponding to starting the diffusion at $y\in M$.

\begin{thm}[Theorem 4.2 of \cite{BP}]\label{liftp}
For $y\in\tilde M$, define the probability measure $\tilde P_y$ on $\Omega_{\tilde M}$ by 
\begin{align*}
	\tilde P_y[A] = P_{q(y)}[\{\omega\mid(y,\omega)\in A\}],
	\quad A\in\mathcal{B}(\Omega_{\tilde M}).
\end{align*}
Then $\tilde P_x$ is the probability measure on $\Omega_{\tilde M}$
for the $\tilde L$-diffusion on $\tilde M$ starting at $x$.
\end{thm}

Let now $X\subseteq M$ be a $*$-recurrent discrete subset and $(F_x,V_x)_{x\in X}$ be regular LS-data for $X$
as in \cref{sublsdh} such that the $V_x$ are connected and evenly covered by $q$.
Let $\tilde X=q^{-1}(X)$.
For $x\in\tilde X$,
let $\tilde V_x$ be the connected component of $q^{-1}(V_{q(x)})$ containing $x$
and $\tilde F_x=\tilde V_x\cap q^{-1}(F_{q(x)})$.

\begin{lem}\label{covdat}
The discrete subset $\tilde X\subseteq\tilde M$ is $*$-recurrent
and the family $(\tilde F_x,\tilde V_x)_{x\in\tilde X}$ is regular LS-data for $\tilde X$.
\end{lem}

\begin{proof}
Since $q$ is a covering and $\tilde L=q^*L$,
the family of $\tilde F_x\subseteq\tilde V_x$ satisfies \ref{d1}, \ref{d2}, and \ref{d4},
where $C$ is the Harnack constant of $L$ from \ref{d4}.
Moreover, the union $\tilde F=\cup_{x\in\tilde  X}\tilde F_x=q^{-1}(F)$ is closed.
Finally,
by the correspondence between the $\tilde L$-diffusion starting at $y\in\tilde M$
and the $L$-diffusion starting at $q(y)\in M$ established in \cref{liftp}
and since the latter hits $F$ with probability one,
we conclude that the first hits $\tilde F$ with probability one.
Hence $\tilde F$ is recurrent.
\end{proof}

\begin{prop}\label{covmea}
The LS-measures $\mu$ and $\tilde\mu$ associated to the families $(F_x,V_x)_{x\in X}$
and $(\tilde F_y,\tilde V_y)_{y\in\tilde  X}$ as above satisfy
\begin{align*}
	\mu_{q(y)}(u) = \sum_{v\in q^{-1}(u)} \tilde\mu_y(v),
	\quad\text{for any $y\in\tilde M$ and $u\in X$.}
\end{align*}
\end{prop}

\begin{proof}
Let $\mu$ and $\tilde\mu$ be finite measures on $M$ and $\tilde M$,
respectively.
Recall the splitting $\mu=\mu'+\mu''$ from \eqref{lsmm},
and suppose that $q_*\tilde\mu=\mu$.
Then, by \cref{liftp} and since $\tilde F=q^{-1}(F)$,
we conclude that $q_*\beta_{\tilde\mu}^{\tilde F} = \beta_\mu^F$.
We get, therefore, that
\begin{align*}
	q_*\tilde\mu' = \mu'
	\quad\text{and}\quad
	q_*\tilde\mu'' = \mu''.
\end{align*}
It follows that, in the recursive construction in \eqref{mm2} and \eqref{mmn},
applied to $y\in\tilde M$ and $q(y)\in M$, respectively,
we have
\begin{align*}
	q_*\tilde\mu_{y,n} = \mu_{q(y),n}
	\quad\text{and}\quad
	q_*\tilde\tau_{y,n} = \tau_{q(y),n}
\end{align*}
for all $n\ge0$.
We conclude that $q_*\tilde\mu_y=\mu_{q(y)}$, which is the assertion.
\end{proof}

Note that \cref{covmea} is a discrete version of the corresponding formula for the transition densities
of the diffusions on $M$ and $\tilde M$ as in \cite[Corollary 4.3]{BP}.

Assume now that we are in the situation of the introduction with an orbifold covering $p\colon M\to M_0$,
where $M$ is the given manifold with the trivial orbifold structure and $M_0$ is a closed orbifold.
Assume furthermore that $L$ and the volume element on $M$ are pull-backs of a diffusion operator
and a smooth volume element on $M_0$.

Then the universal covering $q\colon\tilde M\to M$ composed with $p$ is the universal orbifold covering of $M_0$.
Moreover, there is a group $\Gamma$,
which acts properly discontinuously on $\tilde M$ such that $M_0=\Gamma\backslash\tilde M$.
More generally, let $q\colon\tilde M\to M$ be any covering such that $M_0=\Gamma\backslash\tilde M$,
where $\tilde M$ is connected
and $\Gamma$ is a group which acts properly discontinuously on $\tilde M$.
Since $M_0$ is compact, any such group is finitely generated.

Choose $x_0\in M_0$ and let $X=p^{-1}(x_0)$.
Let $V_0$ be a connected open subset of $M_0$ which is evenly covered (in the sense of orbifolds)
by $\tilde p=p\circ q$, therefore also by $p$.
Let $F_0\subseteq V_0$ be a compact neighborhood of $x_0$ with smooth boundary.
For $x\in X$, let $V_x$ be the connected component of $p^{-1}(V_0)$ containing $x$
and set $F_x=V_x\cap p^{-1}(F_0)$.

\begin{lem}
The discrete subset $X\subseteq M$ is $*$-recurrent
and the family $(F_x,V_x)_{x\in X}$ is regular LS-data.
\end{lem}

\begin{proof}
Clearly, the family of $F_x\subseteq V_x$ satisfies \ref{d1}, \ref{d2}, and \ref{d4},
where $C$ is the Harnack constant of $L_0$ for $(F_0,V_0)$ (in the sense of orbifolds).
Moreover, the union $F=p^{-1}(F_0)$ is closed.
Since $M_0$ is compact, $F$ is recurrent.
\end{proof}

\begin{thm}\label{covres}
Assume that $\tilde\mu$ has finite $a$-moments.
Then $\mu$ has finite $a$-moments,
the restriction of an $a$-bounded $L$-harmonic function to $X$
is an $a$-bounded $\mu$-harmonic function on $X$,
and the restriction map $\H_a(M,L)\to\H_a(X,\mu)$ is an isomorphism.
\end{thm}

By \cref{exmom},
$\tilde\mu$ has finite $a$-moments for the growth functions $e^{\alpha r}$ for sufficiently small $\alpha>0$,
in particular for any growth function of subexponential growth.
Hence the statement of \cref{covres} implies \cref{thmbg} of the introduction.

\begin{proof}[Proof of \cref{covres}]
The first assertion is clear from \cref{covmea} since $q$ does not increase distances.

Let $f$ be a function on $M$ and $\tilde f$ be its lift to $\tilde M$.
Then $f$ is $L$-harmonic if and only if $\tilde f$ is $\tilde L$-harmonic.
If $f$ is $a$-bounded, then $\tilde f$ is $a$-bounded since $q$ does not increase distances.
Conversely, suppose that $\tilde f$ is $a$-bounded.
Let $x\in M$ and $c$ be a shortest geodesic segment from $x_0$ to $x$ in $M$.
Then the lift of $c$ to $\tilde M$ starting in $y_0$ is a shortest geodesic segment from $y_0$
to a point $y\in q^{-1}(y)$ with $d(y_0,y)=d(x_0,x)$.
Since $\tilde f$ lifts $f$ and is (therefore) constant on the fibers of $q$, we obtain
\begin{align*}
	|f(x)| = |\tilde f(y)| \le C_{\tilde f} a(|y|) = C_{\tilde f} a(|x|).
\end{align*}
Therefore $f$ is $a$-bounded,
and hence lifting defines an isomorphism
between the space of $a$-bounded $L$-harmonic functions on $M$
and the space of $a$-bounded $\tilde L$-harmonic functions on $\tilde M$
which are constant on the fibers of $q$.

Let $h$ now be a function on $X$ and $\tilde h$ be its lift to $\tilde X$.
Then $h$ is $\mu$-harmonic if and only if $\tilde h$ is $\tilde \mu$-harmonic,
by \cref{covmea} and since $\tilde h$ is constant on the fibers of $q$.
Moreover, by the argument above, $h$ is $a$-bounded if and only if $\tilde h$ is $a$-bounded.
Thus lifting  defines an isomorphism
between the space of $a$-bounded $\mu$-harmonic functions on $X$
and the space of $a$-bounded $\tilde\mu$-harmonic functions on $\tilde X$
which are constant on the fibers of $q$.

By \cref{thmba}, the restriction map $\H_a(\tilde M,\tilde L)\to\H_a(\tilde X,\tilde\mu)$
is an isomorphism which is equivariant under the group $\Gamma$ of covering transformations of $\tilde p$.
In particular, it is also equivariant with respect to the smaller group $\Gamma'$ of covering transformations of $q$.
Therefore restriction defines an isomorphism between the corresponding subspaces of $\Gamma'$-invariant functions.
But these are exactly the lifts of functions from $\H_a(M,L)$ and $\H_a(X,\mu)$, respectively.
\end{proof}

\section{Applications}
\label{secap}

In this section, we discuss the proofs of Theorems \ref{cornil}, \ref{thmnil},
and \ref{thmsol} from the introduction.
Recall that we are given a diffusion operator $L$
and a smooth volume element on a non-compact and connected manifold $M$,
which are invariant under a group $\Gamma$ acting properly discontinuously and cocompactly on $M$,
such that $L$ is symmetric on $C^\infty_c(M)$ with respect to the volume element.
In particular, $\Gamma$ is a finitely generated infinite group.

We choose an origin $x_0\in M$ such that the isotropy group of $x_0$ in $\Gamma$ is trivial.
In the case, where $\Gamma$ acts as a group of covering transformations,
any point of $M$ is of this kind.
In the general case,
the set of points in $M$ with trivial isotropy group in $\Gamma$ is open and dense.
Since the isotropy group of $x_0$ in $\Gamma$ is trivial,
the orbit map $\Gamma\to X=\Gamma x_0$ is bijective,
and we use it to identify $\Gamma$ with $X$.
We choose balanced LS-data as in \cref{susebal}
and consider the associated probability measure $\mu$ and random walk on $\Gamma$ as in \cref{suseran}.
If the $L$-diffusion on $M$ is transient or, equivalently,
the $\mu$-random walk on $\Gamma$ is transient, then $\mu$ satisfies the following three properties
(\cref{court}):
\begin{enumerate}[label=(P\arabic*)]
\item\label{mu1}
the support of $\mu$ is all of $\Gamma$;
\item\label{mu2}
$\mu$ is symmetric;
\item\label{mu3}
$\mu$ has finite exponential moment (for some sufficiently small exponent).
\end{enumerate}
In particular, in the transient case,
$\mu$ satisfies the properties required in the articles \cite{MPTY,MY,Pe} of Meyerovitch, Perl, Tointon, and Yadin
so that we may apply their results, using \cref{corbg}.

The $\mu$-random walk on $\Gamma$ is recurrent if and only if $\Gamma$
contains $\Z$ or $\Z^2$ as a subgroup of finite index \cite[Theorem 3.24]{Wo}.
In this case, we use the results of Kuchment and Pinchover \cite{KP} on Schr\"odinger operators
invariant under a properly discontinuous and (then also) free action of $A=\Z^k$.
For this application,
we let $\vf^2{\rm dv}$ be the $A$-invariant volume element on $M$
with respect to which $L$ is symmetric on $C^\infty_c(M)$.
Here $\vf$ is an $A$-invariant positive smooth function on $M$
and ${\rm dv}$ denotes the volume element of the Riemannian metric on $M$ induced by $L$.
Then $L$ is of the form
\begin{align*}
	Lf = \Delta f - 2\la\nabla\ln\vf,\nabla f\ra.
\end{align*}
Furthermore, renormalization with $1/\vf$ as in \cref{suselap} transforms $L$
into the $A$-invariant Schr\"odinger operator $S=\Delta+V$ with potential $V=-\Delta\vf/\vf$.
Since $S$ and $\vf$ descend to a Schr\"odinger operator $S_0$ and a positive $S_0$-harmonic function $\vf_0$,
the bottom of the spectrum of $S_0$ on $M_0=A\backslash M$ is $0$. 
Since $\Z^k$ is Abelian, hence amenable, the bottom of the spectrum of $S$
as an unbounded self-adjoint operator on $L^2(M,{\rm dv})$ is also $0$.
Hence Kuchment-Pinchover's \cite[Theorem 5.3]{KP} applies with their $\Lambda_0=0$.
Notice that the smooth functions $[x_j]$ there are equal to $\pm c(|g_j|+1)$ for some constant $c>0$,
where $x\in gD_0$ with $g=(g_1,\dots,g_k)\in A$ 
and $D_0$ denotes the Dirichlet domain about $x_0$ with respect to $A$.

\begin{proof}[Proof of \cref{cornil}]
Suppose first that the $L$-diffusion process on $M$ is transient.
Then the LS-measure on $\Gamma$ satisfies \ref{mu1}--\ref{mu3}. 
Let $f\in\H^d(M,L)$.
Then the restriction $h$ of $f$ to $\Gamma$ belongs to $H^d(\Gamma,\mu)$,
by \cref{corbg}.
But then $h$ is a polynomial of degree at most $d$
on a finite index subgroup $N$ of $\Gamma$, by \cite[Theorem 1.3]{MPTY}.
Therefore the restriction of $f$ to $N$ satisfies the claimed growth property,
by \cite[Proposition 2.7]{MPTY}.
Hence $f$ satisfies the same growth property, by \cref{mubound}.

Suppose now that the $L$-diffusion process on $M$ is recurrent.
Then the $\mu$-random walk on $\Gamma$ is recurrent,
and hence $\Gamma$ is a finite extension of $A=\Z$ or $A=\Z^2$.
Without loss of generality, we may assume that $\Gamma=A$.
Then, by \cite[Theorem 5.3.3]{KP} and the above renormalization, $f/\vf$ is of the form
\begin{align*}
	\frac{f}{\vf} = \sum_{0\le|j|\le d} [x]^j f_j(x),
\end{align*}
where the $f_j$ are $A$-invariant functions on $M$.
(In our case, $A=\Z$ or $A=\Z^2$, but \cite[Theorem 5.3]{KP} also holds for any $\Z^k$.)
\end{proof}

\begin{proof}[Proof of \cref{thmnil}]
Suppose again first that the $L$-diffusion process on $M$ is transient,
so that the LS-measure on $\Gamma$ satisfies \ref{mu1}--\ref{mu3}. 
Without loss of generality, we may assume that $\Gamma=N$.
Combining \cite[Theorems 1.5, 1.6 and Corollary 1.9]{MPTY} and \cite[Theorem 1.5]{Pe},
we have that $\H^d(N,\mu)$ is of finite dimension with
\begin{align*}
	\dim\H^d(N,\mu) = \dim\P^d(N) - \dim \P^{d-2}(N)
\end{align*}
for all $d\ge0$.
Now $\H^d(M,L)\cong\H^d(N,\mu)$ for all $d\ge0$, by \cref{corbg}.

In the recurrent case,
we have again that $\Gamma$ is a finite extension of $A=\Z$ or $A=\Z^2$.
Via renormalization as above,
the desired formula for the dimension of $\H^d(M,L)$ is now given in \cite[Theorem 5.3.2]{KP}.
\end{proof}

\begin{proof}[Proof of \cref{thmsol}]
Since we may assume that $\Gamma$ does not contain $\Z$ or $\Z^2$
as a subgroup of finite index,
we may assume without loss of generality that the $L$-diffusion on $M$ is transient.
By \cref{corbg}, we have $\H^1(M,L)\cong\H^1(\Gamma,\mu)$.
Furthermore, $\mu$ satisfies \ref{mu1}--\ref{mu3}.
Now $\Gamma$ is virtually solvable.
Hence $\Gamma$ is virtually nilpotent if $\H^1(\Gamma,\mu)$ is of finite dimension,
by \cite[Theorem 1.4]{MY}.
Conversely, if $\Gamma$ is virtually nilpotent,
then $\H^1(\Gamma,\mu)$ is of finite dimension, by \cref{thmnil}.
\end{proof}


\end{document}